\documentclass[12pt,twoside]{amsart}

\usepackage{amsmath}
\usepackage{amssymb}
\usepackage{amscd}
\usepackage{graphics}
\usepackage{graphicx}
\usepackage{epstopdf}
\usepackage{epsfig}

\renewcommand{\bar}{\overline}

\newcommand{\CC}{\mathbb{C}}

\newcommand{\HH}{\mathbb{H}}

\newcommand{\NN}{\mathbb{N}}

\newcommand{\PP}{\mathbb{P}}

\newcommand{\RR}{\mathbb{R}}

\newcommand{\Cv}{\CC_v}

\newcommand{\ints}{{\mathcal O}}

\newcommand{\maxid}{{\mathcal M}}

\newcommand{\calA}{{\mathcal A}}

\newcommand{\calC}{{\mathcal C}}
\newcommand{\calD}{{\mathcal D}}

\newcommand{\calF}{{\mathcal F}}

\newcommand{\calJ}{{\mathcal J}}

\newcommand{\calO}{{\mathcal O}}
\newcommand{\calP}{{\mathcal P}}

\newcommand{\calR}{{\mathcal R}}

\newcommand{\calU}{{\mathcal U}}

\newcommand{\Fat}{\calF}
\newcommand{\Jul}{\calJ}

\newcommand{\PCv}{\PP^1(\Cv)}

\newcommand{\Pk}{\PP^1(k)}

\newcommand{\PKbar}{\PP^1(\overline{K})}

\newcommand{\Ber}{\textup{Ber}}
\newcommand{\HBer}{\HH_{\Ber}}

\newcommand{\PBerk}{\PP^1_{\Ber}}

\DeclareMathOperator{\PGL}{PGL}

\DeclareMathOperator{\charact}{char}

\DeclareMathOperator{\Supp}{Supp}

\DeclareMathOperator{\Jac}{Jac}

\DeclareMathOperator{\diam}{diam}

\newcommand{\Dbar}{\bar{D}}
\newcommand{\DBerk}{\calD}
\newcommand{\DbarBerk}{\bar{\calD}}

\newcommand{\dsps}{\displaystyle}

\theoremstyle{plain}
\newtheorem{thm}{Theorem}[section]

\newtheorem{lemma}[thm]{Lemma}

\newtheorem*{thmA}{Theorem A}
\newtheorem*{thmB}{Theorem B}

\theoremstyle{definition}
\newtheorem{defin}[thm]{Definition}

\newtheorem{claim}[thm]{Claim}

\theoremstyle{remark}
\newtheorem{remark}[thm]{Remark}

\numberwithin{equation}{section}

\title[Non-archimedean Connected Julia Sets]
{Non-archimedean connected Julia sets with branching}
\date{October 1, 2014; revised April 6, 2015}

\address[Bajpai,Benedetto,Kim,Marschall,Onul]
        {Amherst College \\ Amherst, MA 01002}
\address[Chen]
        {University of Illinois\\ Urbana, IL 68101}
\address[Xiao]
        {Brown University \\ Providence, RI 02912}
\author[Bajpai]{Dvij Bajpai}
\email[Bajpai]{dbajpai15@amherst.edu}
\author[Benedetto]{Robert~L. Benedetto}
\email[Benedetto]{rlbenedetto@amherst.edu}
\author[Chen]{Ruqian Chen}
\email[Chen]{rchen40@illinois.edu}
\author[Kim]{Edward Kim}
\email[Kim]{ejkim15@amherst.edu}
\author[Marschall]{Owen Marschall}
\email[Marschall]{omarschall15@amherst.edu}
\author[Onul]{Darius Onul}
\email[Onul]{donul15@amherst.edu}
\author[Xiao]{Yang Xiao}
\email[Xiao]{ynxysunny@gmail.com}

\begin{document}

\newcounter{bean}
\newcounter{sheep}

\begin{abstract}
We construct the first examples of rational functions
defined over a nonarchimedean field with a certain
dynamical property: the Julia set in the Berkovich projective line
is connected but not contained in a line segment.
We also show how to compute the measure-theoretic
and topological entropy of such maps.  In particular,
we give an example for which
the measure-theoretic entropy is strictly
smaller than the topological entropy,
thus answering a question of Favre and Rivera-Letelier.
\end{abstract}

\maketitle

Let $\Cv$ be an algebraically closed field equipped with
a non-archimedean absolute value $|\cdot|_v$.  That is,
$|\cdot|_v:\Cv\to [0,\infty)$ with $|x|_v=0$ if and only if $x=0$,
with $|xy|_v=|x|_v |y|_v$, and satisfying the non-archimedean
triangle inequality
$$|x+y|_v \leq \max\{|x|_v, |y|_v\}.$$
We will assume that $\Cv$ is complete with respect to $|\cdot|_v$;
that is, all $|\cdot|_v$-Cauchy sequences converge.
The \emph{ring of integers} $\ints$ of $\Cv$ is the closed unit
disk
$\ints:= \{x\in\Cv : |x|_v \leq 1\}$,
which forms a subring of $\Cv$ with unique
maximal ideal $\maxid:=\{x\in\Cv : |x|_v < 1\}$.
The \emph{residue field} $k$ of $\Cv$ is the quotient $\ints/\maxid$.

The \emph{degree} of a rational function $\phi\in\Cv(z)$ is
$\deg\phi := \max\{\deg f, \deg g\}$, where $\phi=f/g$ and
$f,g\in\Cv[z]$ are relatively prime polynomials.
For $n\geq 0$, we
denote the $n$-th iterate of $\phi$ under composition
by $\phi^n$; that is, $\phi^0(z)=z$,
and $\phi^{n+1}= \phi\circ \phi^n$.

The rational function $\phi$ acts naturally on the Berkovich
projective line $\PBerk$.
We will discuss $\PBerk$ in more detail
in Section~\ref{sec:berk}; for the moment, we note only that
$\PBerk$ is a certain path-connected compactification of $\PP^1(\Cv)$,
and that many of the extra points correspond to closed disks in $\Cv$.
We also note that 
the set $\HBer:=\PBerk\smallsetminus\PCv$ has a natural metric $d_{\HH}$,
although the metric topology on $\HBer$ is stronger than the topology
inherited from $\PBerk$.
We call $\HBer$ \emph{hyperbolic space}, and $d_{\HH}$
the \emph{hyperbolic metric}.

The dynamical action of $\phi$ partitions $\PBerk$ into two
invariant subsets: the (Berkovich) Julia set $\calJ_{\phi}$,
which is the closed subset on which $\phi^n$ acts chaotically,
and the (Berkovich) Fatou set $\calF_{\phi}$,
which is the (open) complement of $\calJ_{\phi}$.
There are numerous examples of Berkovich Julia sets in the literature;
see, for example, \cite{BR,DF1,DF2,FRL,Kiwi,Riv1,Riv2,Riv3}.
However, in all these examples, the Julia set is either
a single point, a line segment, or a disconnected set
(in which case it necessarily has infinitely many connected components).
In this paper, we give the first examples of Berkovich Julia sets
that are connected but are not contained in a line segment.

The main engine we use to produce our examples is the following theorem.
To state it, we note that
a \emph{finite tree} $\Gamma\subseteq\PBerk$ is
exactly what it sounds like:
a subset of $\PBerk$ homeomorphic to a finite tree.
Similarly, an \emph{interval} $I\subseteq\PBerk$ is
a subset of $\PBerk$ homeomorphic to an interval in $\RR$.
See Section~\ref{sec:berk} for details.

\begin{thmA}
Let $\phi\in\Cv(z)$ be a rational function of degree at least $2$
with Berkovich Julia set $\Jul_{\phi}$.
Let $\Gamma\subseteq\HBer$ be a finite tree,
let $I\subseteq \Gamma$ be a compact interval, and let $x_0\in I$.
Suppose that $I$ can be written as a union of
intervals $I=I_1\cup\cdots\cup I_m$ such that
\begin{enumerate}
\item $x_0\in\Jul_{\phi}$,
\item $\phi^{-1}(x_0)\subseteq \Gamma$,
\item $\bigcup_{n\geq 0} \phi^{-n}(I) \supseteq \Gamma$, and
\item for each $i=1,\ldots, m$, there are integers $b_i\geq 1$
and $c_i\geq 2$ such that
\begin{itemize}
\item $\phi^{b_i}$ maps $I_i$ onto $I$, with
\item $d_{\HH}(\phi^{b_i}(x),\phi^{b_i}(y))=c_i d_{\HH}(x,y)$
for all $x,y\in I_i$.
\end{itemize}
\end{enumerate}
Then $\Jul_{\phi}$ is connected and contains $\Gamma$.

Moreover, if $\Gamma$ is not contained in an interval, then
$\Jul_{\phi}$ has a dense subset of branch points,
i.e., points $x\in\Jul_{\phi}$ for which
$\Jul_{\phi}\smallsetminus\{x\}$ has at least three
connected components.
\end{thmA}

Julia sets satisfying the full conditions of Theorem~A can be
considered dendrites, much like the complex Julia set of
$\phi(z)=z^2+i\in\CC(z)$.
In fact, however, there are maps with connected Julia set having a dense
set of points of infinite branching --- that is,
points $x\in\Jul_{\phi}$ for which
$\Jul_{\phi}\smallsetminus\{x\}$ has infinitely many
connected components.
In Theorem~\ref{thm:infbranch}, we will give a
simple sufficient condition for such
infinite branching to occur.
These dendritic Julia sets are rather more complicated
than those like the complex Julia set of $z^2+i$, for which
there is only finite branching at each branch point.


We will also be interested the entropy
of a map $\phi\in\Cv(z)$ acting on $\PBerk$.
In particular, in \cite{FRL}, Favre and Rivera-Letelier
discuss the measure-theoretic entropy
$h_{\mu}(\phi)$ and topological entropy $h_{\textup{top}}(\phi)$
of $\phi$ acting on $\PBerk$.
(Here, $\mu=\mu_{\phi}$ is a certain natural measure
on $\PBerk$ supported on the Julia set $\Jul_{\phi}$;
see Sections~\ref{sec:Julia} and~\ref{sec:invarmeas} for details.)
They prove that
$$0\leq h_{\mu}(\phi) \leq h_{\textup{top}}(\phi)\leq \log \deg(\phi),$$
and they give examples, some with Julia set a line segment,
where the inequalities are all strict.
Since their example maps have fairly large degree,
in \cite[Question 3]{FRL}, they
ask, among other things, whether there exist
rational functions $\phi\in\Cv(z)$
of degree at most $9$ where the inequalities are strict
and the Julia set is connected.

We answer this question positively
when the residue characteristic $p$ of $\Cv$ is $3$.
(In a separate paper, \cite{BenQ}, the second
author will present an example of degree~4,
in residue characteristic $p=2$, for which
the inequalities are strict and the Julia set is connected.)
In particular, for $p=3$
and $a\in\Cv^{\times}$ satisfying $|3|_v\leq |a|_v<1$,
we will show that
$$\phi(z) = \frac{a z^6 + 1}{a z^6 + z^3 - z}
= 1 + \frac{-z^3 + z + 1}{a z^6 + z^3 - z}$$
has connected Julia set and entropies
$$h_{\mu}(\phi) = \log 2 + \frac{5}{11} \log 3
\approx 1.1925\ldots,
\quad\text{and}\quad
h_{\textup{top}}(\phi) = \log \lambda \approx 1.3496\ldots,$$
where $\lambda\approx 3.8558\ldots$
is the largest real root of the polynomial $t^3 -4t^2 - t + 6$.
To our knowledge, $\phi$ is the first rational function
to appear in the literature which acts on $\PBerk$ with topological
entropy \emph{not} the logarithm of an integer.

To compute these entropies, we will use the following result.


%

\begin{thmB}
Let $\phi\in\Cv(z)$ be a rational function of degree $d\geq 2$,
with Julia set $\Jul_{\phi}$ and invariant measure $\mu$.
Suppose that $\Jul_{\phi}$ is connected
and of finite diameter with respect to
the hyperbolic metric $d_{\HH}$.
Let $\calP$ be a countable partition of $\Jul_{\phi}$
such that for every $U\in\calP$,
\begin{itemize}
\item $U$ is path-connected and is
the union of an open set and a countable set;
\item
there is a set $S_U\subseteq\calP$ such that
$\phi$ maps $U$ bijectively onto $\bigcup_{V\in S_U} V$;
and
\item
for all $x\neq y\in U$,
there is an integer $n\geq 0$ such that
either $\phi^n(x)$ and $\phi^n(y)$ belong to different elements
of $\calP$, or else
$\dsps d_{\HH}\big(\phi^n(x),\phi^n(y)\big) \geq  2 d_{\HH}(x,y)$.
\end{itemize}
Then
\begin{enumerate}
\item
$\calP$ is a one-sided generator
for $\phi:\Jul_{\phi}\to\Jul_{\phi}$
of finite entropy.
\item
$\dsps h_{\mu}(\phi) =
\int_{\PBerk} \log\Big( \frac{d}{\deg_x(\phi)} \Big) \, d\mu(x)$.
\item
Let $\calA\subseteq\calP$ consist of those elements of $\calP$
that are uncountable sets, and
let $T$ be the one-sided topological Markov shift on the symbol space $\calA$,
where symbol $U\in\calA$ can be followed by symbol $V\in\calA$
if and only if $V\in S_U$.
Then $h_{\textup{top}}(\phi) = h_{\textup{Gur}}(T)$,
where $h_{\textup{Gur}}(T)$ is the Gurevich entropy of $T$.
\end{enumerate}
\end{thmB}

The existence of a one-sided generator of finite entropy
is a technical condition which enables
the explicit computation of the entropy; see
the proof of Theorem~B in Section~\ref{sec:markov}.
Meanwhile, 
the Gurevich entropy $h_{\textup{Gur}}(T)$
mentioned at the end of Theorem~B is the analog
of the topological entropy for
a countable-state Markov shift $T$;
see Section~\ref{sec:gurevich}.

The outline of the paper is as follows.
In Section~\ref{sec:berk} we discuss the basic properties of the
Berkovich projective line $\PBerk$.
In Section~\ref{sec:dynamics}, we give some relevant background
on the theory of dynamics on $\PBerk$;
we also include some discussion of entropy, especially
Gurevich entropy.
We prove Theorem~A in Section~\ref{sec:proofa},
and Theorem~B in Section~\ref{sec:markov}.
Finally, in Section~\ref{sec:examples}, we present
our degree six example.

\section{Background on the Berkovich projective line}
\label{sec:berk}

\subsection{Berkovich points and disks}
The Berkovich projective line $\PBerk$ over $\Cv$ is a certain
compact Hausdorff topological space
containing $\PCv$ as a subspace.  The precise definition,
involving multiplicative seminorms on $\Cv$-algebras,
is tangential to this paper; the interested reader may consult
Berkovich's monograph \cite{Ber}, the detailed
exposition in \cite{BR}, or the summary in \cite{BenAZ},
for example.  We give only an overview here, without proofs.

There are four types of points in $\PBerk$.  Type~I points
are simply the points of $\PCv$.  Points of type~II and~III
correspond to closed disks in $\Cv$.  Specifically, for
$a\in\Cv$ and $r>0$, the closed disk
$\Dbar(a,r) := \{z\in \Cv : |z-a|_v\leq r\}$
corresponds to a unique point, which we shall denote
$\zeta(a,r)$, in $\PBerk$.
If $r\in |\Cv^{\times}|_v$, the point $\zeta(a,r)$ is
of type~II; otherwise, $\zeta(a,r)$ is of type~III.
The type~II point $\zeta(0,1)$ is known as the
\emph{Gauss point}.
The remaining points of $\PBerk$, knowns as the type~IV
points, correspond to equivalence classes of
descending chains $D_1\supsetneq D_2 \supsetneq \cdots$
of disks in $\Cv$ with empty intersection.
Type~IV points will not be important in this paper,
though.

The absolute value $|\cdot|_v$ extends to a function
$|\cdot|_v : \PBerk\smallsetminus\{\infty\} \to [0,\infty)$,
where in particular $|\zeta(a,r)|_v:=\max\{|a|_v,r\}$.

The space $\PBerk$ is equipped with a topology,
known as either the \emph{Gel'fand topology} or the
\emph{weak topology}.  The Gel'fand topology restricted to
$\PCv$ coincides with the usual topology induced by $|\cdot|_v$.
Unlike $\PCv$, however, $\PBerk$ is compact (and still Hausdorff).

Given $a\in\Cv$ and $r>0$,
$\DbarBerk(a,r)\subseteq\PBerk$
will denote
a certain closed subset whose
type~I points are those in $\Dbar(a,r)$,
and whose type~II and~III points are of the form
$\zeta(b,s)$ with $\Dbar(b,s)\subseteq\Dbar(a,r)$.
Similarly, $\DBerk(a,r)\subseteq\PBerk$
is a certain open subset whose type~I points are
those in $D(a,r)\subseteq\Cv$,
and whose type~II and~III points are of the form
$\zeta(b,s)$ with $\Dbar(b,s)\subseteq D(a,r)$
\emph{other than} $\zeta(a,r)$ itself.

More generally,
a \emph{closed Berkovich disk} is a subset of $\PBerk$
of the form either $\DBerk(a,r)$ or
$\PBerk\smallsetminus\DbarBerk(a,r)$.
Similarly, an
\emph{open Berkovich disk} is a subset of $\PBerk$
of the form either $\DbarBerk(a,r)$ or
$\PBerk\smallsetminus\DBerk(a,r)$.
For each of these four possibilities, the boundary
of the Berkovich disk consists of the single point $\zeta(a,r)$.
Meanwhile, a \emph{closed} (respectively, \emph{open})
\emph{connected affinoid} $U$ is a nonempty finite intersection
of closed (respectively, open) Berkovich disks.
If none of the disks in the intersection contains any other,
then the boundary of $U$ consists of the
(finitely many) boundary points of the disks in the intersection.

\subsection{Path-connectedness and the hyperbolic metric}
The space $\PBerk$ is uniquely path-connected.
That is, for any
$x\neq y\in\PBerk$, there is a unique subspace
$I\subseteq\PBerk$ homeomorphic to the unit interval
$[0,1]\subseteq\RR$,
where $x,y\in I$, and the homeomorphism takes
$x$ to $0$ and $y$ to $1$.
The set $I$ is called an \emph{interval} in $\PBerk$,
and it is denoted $[x,y]$ or $[y,x]$.
The points in $[x,y]$ are said to \emph{lie between}
$x$ and $y$.
For example, if $x=\zeta(a,r)$ and $y=\zeta(a,s)$ with $s>r>0$,
then the interval $[x,y]$ consists of all points
$\zeta(a,t)$ with $r\leq t\leq s$.
$\PBerk$ is also locally path-connected.  Thus, any connected
open subset is, like $\PBerk$ itself, uniquely path-connected.

Given a set of points $S\subseteq\PBerk$,
the \emph{convex hull} of $S$ is the set
$\Gamma:=\bigcup_{x,y\in S} [x,y]$
consisting of all points lying between two points of $S$.
If $S$ is finite and nonempty, then the convex hull $\Gamma$ is a
finite tree.  Conversely, any finite tree $\Gamma\subseteq\PBerk$
is the convex hull of some nonempty finite set $S\subseteq\PBerk$.

The unique path-connectedness endows $\PBerk$
with a tree-like structure; however, unlike a finite tree,
it has points of infinite branching, and such points
are dense in any interval $[x,y]$ with $x\neq y$.
In fact, these infinitely-branched points are precisely the
type~II points, because for any type~II point $x=\zeta(a,r)$,
the complement
$\PBerk\smallsetminus\{x\}$
consists of infinitely many connected components,
called the \emph{residue classes} at $x$.
Each of these components is an open Berkovich disk:
one is $\PBerk\smallsetminus\DbarBerk(a,r)$, while the rest are
the infinitely many disks
$\DBerk(b,r)\subseteq\PKbar$ with $b\in\Dbar(a,r)$.
Thus, the residue classes at $x$
are in one-to-one correspondence with the points
of $\Pk$, where $k$ is the residue field of $\Cv$.

On the other hand, if $\zeta(a,r)\in\PBerk$ is of type~III,
then $\PBerk\smallsetminus\{\zeta(a,r)\}$ has two components:
the open Berkovich disks $\DBerk(a,r)$
and $\PBerk\smallsetminus\DbarBerk(a,r)$.
In addition, if $x\in\PBerk$ is either of type~I or type~IV,
then $\PBerk\smallsetminus\{x\}$ is still connected.

The set $\HBer:=\PBerk\smallsetminus\PCv$ is called
\emph{hyperbolic space} and admits a metric $d_{\HH}$, which
can be easily described on the type~II and~III points.
Given $a\in\Cv$ and $r\geq s>0$, we have
$$d_{\HH}\big(\zeta(a,r) , \zeta(a,s)\big) := \log r - \log s.$$
The hyperbolic distance between two points $\zeta(a,r)$
and $\zeta(b,s)$ is then the sum of the distances from each
point to $\zeta(c,t)$, where $\Dbar(c,t)$ is the smallest closed
disk in $\Cv$ containing both $\Dbar(a,r)$ and $\Dbar(b,s)$.
The reader should be warned, however, that the $d_{\HH}$-metric
topology on $\HBer$ is strictly stronger than Gel'fand topology,
and hence the former is known as the \emph{strong topology}.
However, if $\Gamma\subseteq\HBer$ is a finite tree, then the Gel'fand
and strong topologies each induce the same subspace topology on $\Gamma$.

\subsection{Rational functions on $\PBerk$}
\label{sec:ratact}
Any nonconstant rational function $\phi\in\Cv(z)$
extends uniquely to a continuous function
from $\PBerk$ to $\PBerk$.  This extension, which we
also denote $\phi$, is in fact an open map.
For each point $x\in\PBerk$, the image $\phi(x)$ is a
point of the same type.  In addition, $\phi$ has a local
degree $\deg_x(\phi)$ at $x$.
The local degree is an integer between $1$ and $\deg\phi$;
in fact, for any $y\in\PBerk$, we have
\begin{equation}
\label{eq:degsum}
\sum_{x\in\phi^{-1}(y)} \deg_x(\phi) = \deg\phi.
\end{equation}
That is, every point of $\PBerk$ has exactly $\deg\phi$ preimages,
counted with multiplicity.
Rather than formally defining $\phi$ and $\deg_x(\phi)$ on $\PBerk$,
we note the following special cases.
\begin{enumerate}
\item
If $D(a,r)\subseteq\Cv$ is an open disk
containing no poles of $\phi$, then the image $\phi(D(a,r))$
is also an open disk $D(b,s)$, and we have
$\phi(\zeta(a,r))=\zeta(b,s)$.
If there are no poles in the closed disk $\Dbar(a,r)$,
then $\phi:\Dbar(a,r)\to\Dbar(b,s)$ is everywhere $m$-to-$1$,
counting multiplicity, for some $m\geq 1$;
in that case, $\deg_{\zeta(a,r)}(\phi)=m$.
\item
Writing $\phi=f/g$ with $f,g\in\ints[z]$ and at least one coefficient
of $f$ or $g$ in $\ints^{\times}$, define
$$\overline{\phi}:= \bar{f} / \bar{g} \in k(z)\cup\{\infty\}.$$
We have $\phi(\zeta(0,1))=\zeta(0,1)$ if and only if $\bar{\phi}$
is not constant.
In that case, $\deg_{\zeta(0,1)}(\phi) = \deg(\bar{\phi})$.
\item Given $\psi\in\Cv(z)$ and $w\in\Cv$, we will write
\begin{equation}
\label{eq:approxdef}
\phi(w)\approx \psi(w)
\quad\text{if}\quad |\phi(w)-\psi(w)|_v< |\phi(w)|_v.
\end{equation}
Fix $a\in\Cv$ and $r>0$. Suppose there is some $\psi\in\Cv(z)$
and $b\in\Cv$ such that
\begin{itemize}
\item $\phi(w)-b\approx \psi(w)$ for all $w\in\Cv$ with $|w-a|_v<r$, and
\item $\psi(\zeta(a,r))=\zeta(0,s)$ for some $s>0$.
\end{itemize}
Then
$\phi(\zeta(a,r)) = \zeta(b,s)$,
and $\deg_{\zeta(a,r)}(\phi) = \deg_{\zeta(a,r)}(\psi)$.
\end{enumerate}

If $U\subseteq\PBerk$ is connected and $\deg_x(\phi)=n$
is constant on $U$, then
$$d_{\HH}(\phi(x),\phi(y)) = n d_{\HH}(x,y) \quad\text{for all } x,y\in U.$$

%

\section{Dynamics on the Berkovich projective line}
\label{sec:dynamics}

\subsection{Berkovich Julia sets}
\label{sec:Julia}

Just as in complex dynamics, a rational function $\phi\in\Cv(z)$
has an associated Julia set.  As in Section~\ref{sec:berk},
we will only give definitions and state properties without proofs.
For further details,
see, for example,
\cite[Section~6.4]{BenAZ},
\cite[Section~10.5]{BR}, or \cite[Section~2.3]{FRL}.

\begin{defin}
\label{def:Julia}
Let $\phi\in\Cv(z)$ be a rational function of degree $d\geq 2$.
The (\emph{Berkovich}) \emph{Julia set} of $\phi$ is the set
$\Jul_{\phi}\subseteq \PBerk$ consisting of those points $x\in\PBerk$
with the property that for any open set $U\subseteq\PBerk$
containing $x$,
$$\PBerk \smallsetminus \Big[ \bigcup_{n\geq 0} \phi^{-n}(U) \Big]
\quad \text{is finite}.$$
\end{defin}

The Berkovich Julia set of $\phi\in\Cv(a)$ has a number
of properties familiar to complex dynamicists.
First, it is closed in $\PBerk$ and hence compact.
Second, it is invariant under $\phi$, in the sense that
$\phi^{-1}(\Jul_{\phi}) = \phi(\Jul_{\phi})=\Jul_{\phi}$.
Third, for any $x\in\Jul_{\phi}$, the backward orbit
$$\calO^{-}_{\phi}(x) := \bigcup_{n\geq 0} \phi^{-n}(x)$$
is a dense subset of $\Jul_{\phi}$.
Fourth, if $x=\zeta(a,r)$ is a repelling periodic point
--- i.e., if $\phi^n(x)=x$ and $\deg_x(\phi^n)\geq 2$
for some $n\geq 1$ --- then $x\in\Jul_{\phi}$.

Meanwhile, it is possible that the $\Jul_{\phi}$
consists of a single point $x$,
which is necessarily a fixed point of type~II
satisying $\deg_x(\phi)=\deg\phi$.  In that case, $\phi$
is said to have \emph{potentially good reduction}.  However,
we will be mainly concerned with maps $\phi$
that do \emph{not} have potentially good reduction,
i.e., for which $\Jul_{\phi}$ consists of more than one point.
In that case,
$\Jul_{\phi}$ is a perfect set and in particular is
uncountable.

We note two other useful properties of the Julia set.
First,
for any open set $U$ intersecting $\Jul_{\phi}$,
some iterate $\phi^n(U)$ contains $\Jul_{\phi}$,
by \cite[Theorem~10.56(B')]{BR}.
Second, $\Jul_{\phi}$ is a
separable metrizable space and hence is second countable;
in fact, it has a basis consisting of countably many open
connected affinoids intersected with $\Jul_{\phi}$.
(This is true by \cite[Lemma~7.15]{FJ}, because, as
we will see in Section~\ref{sec:invarmeas},
$\Jul_{\phi}$ is the support of a Borel measure.)
It follows that $\Jul_{\phi}$
can intersect only countably many residue classes
of any point $x\in\PBerk$.

\subsection{The invariant measure}
\label{sec:invarmeas}

Again in parallel with complex dynamics,
given a rational function $\phi\in\Cv(z)$ of degree $d\geq 2$,
there is a naturally associated Borel probability measure $\mu=\mu_{\phi}$
on $\PBerk$.
(The construction of this measure, via potential theory
on $\PBerk$, first appeared in \cite{BRold,FRLold};
see also \cite[Section~10.1]{BR} and \cite{FRL}.)
The measure $\mu$ is invariant with respect to $\phi$,
meaning that $\mu(U)=\mu(\phi^{-1}(U))$ for any Borel set
$U\subseteq\PBerk$.
According to \cite[Th\'{e}or\`{e}me~A]{FRL}, 
it is also mixing and hence ergodic, so that
any Borel set $U$ with $\phi^{-1}(U)=U$ has either
$\mu(U)=0$ or $\mu(U)=1$.
By the same theorem, its support $\Supp\mu$ is precisely
$\Jul_{\phi}$; see also \cite[Sections~10.1, 10.5]{BR}.

In the context of ergodic theory, the \emph{Jacobian}
of $\phi:\PBerk\to\PBerk$ is a nonnegative-valued function
such that for every Borel set $U\subseteq\PBerk$
on which $\phi$ is injective,
we have $\mu(\phi(U)) = \int_U \Jac_{\phi}(x) \, d\mu(x)$.
In \cite[Lemme~4.4(2)]{FRL}, Favre and Rivera-Letelier
proved that any rational function $\phi\in\Cv(z)$
of degree $d\geq 2$ has a Jacobian, given by
\begin{equation}
\label{eq:jacfmla}
\Jac_{\phi}(x) = \frac{d}{\deg_x(\phi)}.
\end{equation}

\subsection{Entropy}
\label{sec:entropy}

Let $X$ be a topological space and $f:X\to X$ a continuous function.
If $\mu$ is a Borel probability measure on $X$ and $\calP$ is
a finite partition of $X$ into measurable pieces,
the entropy of $\calP$ is
$H(\calP) := -\sum_{U\in\calP} \mu(U)\log(\mu(U))$.
If $\mu$ is $f$-invariant, the
(\emph{measure-theoretic}) \emph{entropy} of $f$ is
$$h_{\mu}(f) = h_{\mu}(X,f)
:= \sup_{\calP} \lim_{n\to\infty} \frac{1}{n}
H(\calP \vee f^{-1} \calP \vee \cdots \vee f^{-n}\calP),$$
where the supremum is over all finite measurable partitions $\calP$
of $X$, the partition
$\calP\vee\calP'$ consists of all nonempty intersections
$U\cap U'$ with $U\in\calP$ and $U'\in\calP'$,
and $f^{-i}\calP := \{f^{-i}(U) : U\in\calP\}$.

If $X$ is compact, the \emph{topological entropy} of $f$ is
$$h_{\textup{top}}(f) = h_{\textup{top}}(X,f) =
\sup_{\calU} \lim_{n\to\infty} \frac{1}{n}
\log N(\calU \vee f^{-1} \calU \vee \cdots \vee f^{-n}\calU),$$
where the supremum is over all finite open covers of $X$, and
$N(\calU)$ is the minimum number of elements of $\calU$ needed
to cover $X$.  
If $X$ is compact and metrizable,
the \emph{variational principle} states that
$h_{\textup{top}}(f) = \sup\{h_{\mu}(f):\mu\in M(X,f)\}$,
where $M(X,f)$ is
the set of all $f$-invariant Borel probability measures on $X$.
(See, for example, \cite[Theorem~8.6]{Wal}.)

Let $\phi\in\Cv(z)$ be a rational function of degree $d\geq 2$,
with Julia set $\Jul_{\phi}\subseteq\PBerk$,
invariant measure $\mu=\mu_{\phi}$ supported on $\Jul_{\phi}$,
and Jacobian function $\Jac_{\phi}(x) = d/\deg_x(\phi)$
as in Section~\ref{sec:invarmeas}.
In \cite[Th\'{e}or\`{e}mes~C,D]{FRL}, Favre and Rivera-Letelier proved
that
$$0\leq \int_{\PBerk} \log \Jac_{\phi}(x) \, d\mu(x)
\leq h_{\mu}(\Jul_{\phi},\phi)
\leq h_{\textup{top}}(\Jul_{\phi},\phi)
= h_{\textup{top}}(\PBerk,\phi)
\leq \log d.$$
(Actually, if $\charact\Cv=p>0$,
then the final $d$ above can
be replaced by the sharper $\deg_{\textup{sep}}(\phi)$,
the smallest degree of $\psi\in\Cv(z)$ such that we can write
$\phi(z)=\psi(z^{p^r})$.  Since all of our maps will be separable,
however, we will always have $\deg_{\textup{sep}}(\phi)=d$.)

\subsection{Gurevich Entropy}
\label{sec:gurevich}

Let $\calA$ be a countable set, and let $C=\{c(a,b)\}_{a,b\in\calA}$
be an $\calA\times\calA$ matrix of $1$'s and $0$'s.
Equip $\calA$ with the discrete topology, let $X=\calA^{\NN}$
equipped with the product topology, and let
$$Y :=\big\{ \{a_n\}_{n\geq 0}\in X : c(a_n,a_{n+1})=1
\text{ for all } n\geq 0\big\}$$
be the set of sequences consistent with $C$.
Define the (\emph{one-sided}) \emph{topological Markov shift}
$T:Y\to Y$ by $T(\{a_n\}_{n\geq 0}) = \{a_{n+1}\}_{n\geq 0}$.
If $\calA$ is finite, then $Y$ is compact, and the variational
principle holds for $h_{\textup{top}}(Y,T)$.
However, if $\calA$ is countable, then $Y$ is
usually not compact, making it unclear how to even define
the topological entropy.  Gurevich \cite{Gur69} addressed this issue
by defining the \emph{Gurevich entropy} to be
$$h_{\textup{Gur}}(T)=h_{\textup{Gur}}(Y,T)
=\sup h_{\textup{top}}(Y',T'),$$
where the supremum is over all subsets $Y'\subseteq Y$
(with associated shift $T':=T|_{Y'}$) formed by restricting
to a finite subset $\calA'\subseteq\calA$ of symbols.

The Gurevich entropy can often be computed combinatorially.
For $a,b\in\calA$, a \emph{path} of length $n$ from $a$ to $b$ is
a finite sequence $a=a_0,a_1,\ldots,a_n=b$ with $c(a_{i-1},a_i)=1$
for all $i=1,\ldots,n$.
For any $n\geq 0$, let
$p_{a,b}(n)$ denote the number of paths of length $n$ from $a$ to $b$.
Let $R$ be the radius of convergence (convergence in $\CC$,
that is) of the associated generating function
$P_{a,b}(z) := \sum_{n\geq 0} p_{a,b}(n) z^n$.
If the underlying directed graph is strongly connected --- i.e., if for
any $a,b\in\calA$, there is a path from $a$ to $b$ --- then
Vere-Jones proved \cite{VJ} that $R$ is independent of $a$ and $b$,
and Gurevich proved \cite{Gur70} that $h_{\textup{Gur}}(T)=-\log R$.

For any $a\in\calA$, a \emph{first-return loop} at $a$ is a
path from $a$ to $a$ for which none of the intermediate symbols
in the path is $a$. For each $n\geq 1$, let $f_a(n)$ be the
number of first-return loops at $a$ of length $n$.
It is easy to show (see, for example, \cite[Lemma~7.1.6(iv)]{Kit}
or \cite[Equation~(1)]{Rue}) that the associated generating
function $F_a(z) := \sum_{n\geq 1} f_a(n) z^n$ satisfies
$P_{a,a}(z) = 1/(1-F_a(z))$ for $|z|<R$.
In particular, if $1-F_a(z)$ has a real root $r>0$
such that $F_a$ converges and is nonzero on $|z|<r$,
then $r=R$, and hence $h_{\textup{Gur}}(T)=-\log r$.


For more on Markov shifts on countably many symbols,
see, for example,
\cite{BBG} \cite[Chapter 1]{GurSav}, \cite[Chapter 7]{Kit},
\cite{Rue}, or \cite{Sar}.


\section{Proof of Theorem~A and Related Results}
\label{sec:proofa}

\begin{lemma}
\label{lem:treepreim}
Let $\phi\in\Cv(z)$ be a nonconstant rational function,
and let $\Gamma\subseteq\HBer$ be a finite tree.  Then $\phi^{-1}(\Gamma)$
is the union of finitely many pairwise disjoint finite trees
$Z_1,\ldots,Z_n$.  Moreover, for each $i=1,\ldots, n$,
we have $\phi(Z_i)=\Gamma$.
\end{lemma}

\begin{proof}
Pick $y_1,\ldots,y_m\in\HBer$ such that $\Gamma$ is the convex hull
of $S:=\{y_1,\ldots,y_m\}$; discarding some of these points if necessary,
we may assume that no $y_i$ lies between any other two $y_j$'s.
Let $W$ be the unique connected component
of $\PBerk\smallsetminus S$ intersecting $\Gamma$.
Note that $W$ is an open connected affinoid (possibly with finitely
many type~IV points removed, if some $y_i$ is of type~IV), with
$\partial W =S$.
By \cite[Lemma~9.12]{BR}, $\phi^{-1}(W)$ is a union of at most
$d$ pairwise disjoint open connected affinoids $U_1,\ldots, U_n$
(again, possibly with finitely many type~IV points removed),
where $\phi(U_i)=W$
and $\phi(\partial U_i)=S$, for each $i=1,\ldots, n$.

For each $i$, let $Z_i := \overline{U}_i \cap \phi^{-1}(\Gamma)$.
For any $z_1,z_2\in Z_i$, the image $\phi((z_1,z_2))$
of the open interval strictly between them is a path between
$\phi(z_1),\phi(z_2)\in \Gamma$ that never hits $S$.
Thus, $\phi([z_1,z_2])\subseteq \Gamma$, and hence $[z_1,z_2]\subseteq Z_i$.
That is, $Z_i$ is connected.  Moreover, since $\phi$ is an open
map, the only points $w\in Z_i$ for which $Z_i\smallsetminus \{w\}$
is still connected
are those for which $\phi(w)$ is an endpoint of $\Gamma$.  In other
words, the only endpoints of $Z_i$ belong to
$\phi^{-1}(S)\cap\overline{U}_i = \partial U_i$.
Since each $y_i$ has at most $\deg\phi$ preimages,
$\partial U_i$ is a finite set, and hence $Z_i$ is a finite tree.
Finally, because $\phi(\overline{U_i})=W$, we have $\phi(Z_i)=\Gamma$.
\end{proof}

\begin{proof}[Proof of Theorem~A]
\textbf{Step 1}.
First, we will show
that the backward orbit $\bigcup_{n\geq 0} \phi^{-n}(x_0)$
intersects $I$ as a dense subset.
Let $L$ be the (hyperbolic)
length of $I$, i.e., the $d_{\HH}$-distance between
the two endpoints of $I$.
By hypothesis~(d), each subinterval $I_i\subseteq I$
has hyperbolic length $c_i^{-1}L \leq L/2$
and contains an element of $\phi^{-b_i}(x_0)$.

Partition each $I_i$ into $m$ subintervals
$I^{(2)}_{i,j}:= \phi^{-b_i}(I_j)\cap I_i$,
for $1\leq j\leq m$.  Each $I^{(2)}_{i,j}$ must have length
at most $L/4$ and contain an element of 
$\phi^{-b_i-b_j}(x_0)$.
Continuing in this fashion, we obtain, at the $\ell$-th step,
a partition of $I$ into $m^{\ell}$ subintervals, each of length
at most $L/2^{\ell}$ and containing an elment of the backward
orbit of $x_0$.  Letting $\ell$ increase to infinity, we see
that $I\cap \bigcup_{n\geq 0} \phi^{-n}(x_0)$ is dense in $I$,
as claimed.

\textbf{Step 2}.
Next, we observe that $\Gamma\subseteq \Jul_{\phi}$.  Indeed,
since $x_0\in \Jul_{\phi}$ by hypothesis~(a), and since the Julia set is
closed and invariant under $\phi$, it follows from Step~1 that
$I\subseteq\Jul_{\phi}$.  By hypothesis~(c) and the invariance
of $\Jul_{\phi}$, then, we have $\Gamma\subseteq\Jul_{\phi}$.

\textbf{Step 3}.
Now we show that $\Jul_{\phi}$ is connected.
For every integer $n\geq 0$, let $\Gamma_n:=\phi^{-n}(\Gamma)$,
and let $X_n:=\Gamma_0 \cup \cdots \cup \Gamma_n$.
Then $X_n\subseteq\Jul_{\phi}$, again by the
invariance of the Julia set.

We will show, by induction, that $X_n$ is a finite tree.
By hypothesis, $X_0=\Gamma$ is a finite tree.  If we know $X_n$
is a finite tree, then by Lemma~\ref{lem:treepreim},
$\phi^{-1}(X_n)$ is a finite union
$Z_1\cup\cdots\cup Z_M$ of finite trees, each of which
maps onto $X_n$ under $\phi$.
By hypothesis~(b), then,
since $x_0\in \Gamma\subseteq X_n$, each tree $Z_i$
has a nontrivial intersection with $\Gamma$,
including at least one point of $\phi^{-1}(x_0)$.
Hence, $X_{n+1} = X_n \cup \Gamma$ is a finite union of finite trees,
each of which intersects $\Gamma\subseteq X_{n+1}$.
Since $X_{n+1}\subseteq\PBerk$ has no loops, it is indeed 
a finite tree.

Thus, $X:=\bigcup_{n\geq 0}X_n=\bigcup_{n\geq 0}\Gamma_n \subseteq\Jul_{\phi}$
is connected.
Since $\bigcup_{n\geq 0}\phi^{-n}(x_0)\subseteq X$
is dense in $\Jul_{\phi}$,
then, $\Jul_{\phi}=\overline{X}$ is the closure of a
connected set and hence is connected.

\textbf{Step 4}.
If $\Gamma$ is not contained in an interval, there is a point $y\in \Gamma$
such that $\Gamma\smallsetminus\{y\}$ has at least three components.
Because $\PBerk$ is uniquely path-connected,
$y$ is also a branch point of $\Jul_{\phi}\supseteq \Gamma$.
By the invariance of $\Jul_{\phi}$ and the mapping properties of
residue classes, then, any point $x$ in the backward orbit of $y$
is also a branch point of $\Jul_{\phi}$.
Since $\bigcup_{n\geq 0}\phi^{-n}(y)$ is dense in $\Jul_{\phi}$,
we are done.
\end{proof}

As promised in the introduction, we can prove in some cases
that $\Jul_{\phi}$ has a dense subset of points of
\emph{infinite} branching, i.e.,
points $x\in\Jul_{\phi}$
for which $\Jul_{\phi}\smallsetminus\{x\}$ has infinitely many
connected components.
To state the relevant condition,
we need the following definition.

\begin{defin}
Let  $p:=\charact k$ be the residue characteristic of $\Cv$,
let $\phi\in\Cv(z)$ be a rational function, and let
$x\in\PBerk$ be a point of type~II that is
periodic of period $m\geq 1$.
Let $h\in\PGL(2,\Cv)$ map $x$ to $\zeta(0,1)$,
and let $\psi:=h\circ\phi^m \circ h^{-1}$,
so that the reduction $\overline{\psi}\in k(z)$
has degree $\deg\overline{\psi}\geq 1$.
If $\overline{\psi}$ is purely inseparable,
i.e., if $\overline{\psi}(z)=\eta(z^{p^r})$
for some $\eta\in\PGL(2,k)$ and some $r\geq 0$,
we say that $x$ is a \emph{purely inseparable}
periodic point of $\phi$.
\end{defin}

\begin{thm}
\label{thm:infbranch}
Let $\phi$, $\Gamma$, and $I$ be as in Theorem~A.
Suppose that $\Gamma$ has a branch point $y$ that is
periodic under $\phi$.
Suppose also that $y$ is not a purely inseparable
repelling periodic point.
Then $\Jul_{\phi}$ has a dense set of points $x$
at which $\Jul_{\phi}$ has infinite branching.
\end{thm}

\begin{proof}
Replacing $\phi$ by an iterate, we may assume $y$ is fixed.
Since $y$ has at least three residue classes, it must be of type~II.
After a change of coordinates, then, we may assume that
$y=\zeta(0,1)$ is the Gauss point.  Let $\bar{\phi}\in k(z)$ denote
the reduction of $\phi$, so that $\deg(\bar{\phi})\geq 1$.
We claim that $\zeta(0,1)$ is a point of infinite branching of
$\Jul_{\phi}$.

If $\deg(\bar{\phi})=1$, then since $\zeta(0,1)\in\Jul_{\phi}$,
there must be a residue class $U$ of $\zeta(0,1)$ that intersects
$\Jul_{\phi}$ but such that the corresponding point $u$ of $\Pk$
is not periodic under $\bar{\phi}$.
(See \cite[Theorem~6.19]{BenAZ} and the discussion that follows it.)
The backward orbit
$\bigcup_{n\geq 0} \bar{\phi}^{-n}(u)$ in $\Pk$ is therefore
infinite, and hence there are infinitely many corresponding residue
classes $U'$ of $\zeta(0,1)$ whose forward orbits contain $U$.
These infinitely many residue classes therefore all intersect $\calJ_{\phi}$.

On the other hand,
if $\deg(\bar{\phi})\geq 2$, then $\zeta(0,1)$ is a repelling
fixed point;
by hypothesis, then, $\bar{\phi}$ is not purely inseparable.
Thus, there are at most two points of $\Pk$
that have finite backward orbit under $\bar{\phi}$.
(Outside of the purely inseparable case,
the Riemann-Hurwitz formula applies, bounding
the number of such exceptional points, just as in
complex dynamics.)
Since there are at least three directions
at $\zeta(0,1)$ intersecting $\Jul_{\phi}$, there must be at
least one such direction $U$ for which the corresponding point
$u\in\Pk$ has infinite backward orbit under $\bar{\phi}$.
As in the previous paragraph,
the infinitely many corresponding preimage
residue classes must all intersect $\calJ_{\phi}$.

Either way, then, infinitely many residue classes
at $\zeta(0,1)$ intersect $\calJ_{\phi}$.
That is, $y$ is a point of infinite branching of
$\Jul_{\phi}$.  As in Step~4 of the proof of Theorem~A, then,
all elements of the backward orbit of $y$ are also points of
infinite branching of $\Jul_{\phi}$.
Since $\bigcup_{n\geq 0} \phi^{-n}(y)$ is dense in $\Jul_{\phi}$,
we are done.
\end{proof}

\section{Computing the entropy}
\label{sec:markov}

The heart of Theorem~B is conclusion (a),
that the countable partition $\calP$
is a one-sided generator of finite entropy.
Being a one-sided generator for $\phi$ means that
the $\sigma$-algebra generated by
$\big\{ \phi^{-n}(U) : U\in\calP \text{ and } n\geq 0 \big\}$
is the full Borel $\sigma$-algebra on $\Jul_{\phi}$.
The finiteness of the entropy, meanwhile, is by the following lemma.

\begin{lemma}
\label{lem:finent}
Let $\phi\in\Cv(z)$, $\mu$, and $\calP$ be as in Theorem~B.
Then
$$\sum_{U\in\calP} -\mu(U) \log(\mu(U)) < \infty.$$
\end{lemma}

\begin{proof}
If $\calP$ is finite, the statement is trivial.  Thus,
we may assume that $\calP$ is infinite, and hence
that $\Jul_{\phi}$ has more than one point.
As noted in Section~\ref{sec:Julia}, $\Jul_{\phi}$
is therefore uncountable.  Since each
of the countably many elements
of $\calP$ either is countable or contains a nonempty 
open set, at least one $U_0\in\calP$ must
contain a nonempty open set.
By properties of the Julia set,
there must be some $N\geq 1$ such that
$\phi^N(U_0)=\Jul_{\phi}$.

For every integer $n\geq 0$, define
$$\calP_n :=\{U\in\calP : \phi^i(U)\cap U_0 = \varnothing
\text{ for all } i=0,\ldots, n-1\},$$
so that $\calP=\calP_0 \supseteq \calP_1\supseteq \cdots$.
Next, define $A_n:=\bigcup_{U\in \calP_n} U$,
so that $\Jul_{\phi}=A_0 \supseteq A_1 \supseteq \cdots$.

\textbf{Step 1}.
We claim that
$A_{n+N} \subseteq \phi^{-N}(A_n) \cap A_N$
for any $n\geq 0$.
Clearly $A_{n+N}\subseteq A_N$.  Given $U\in\calP_{n+N}$,
then, it suffices to show that $\phi^N(U)\subseteq A_n$.
By the hypotheses of Theorem~B, $\phi^N(U)=\bigcup_{V\in \calR} V$
for some subset $\calR\subseteq\calP$.
Since
$$\phi^i(V)\cap U_0 \subseteq \phi^i(\phi^N(U))\cap U_0
= \phi^{N+i}(U) \cap U_0 = \varnothing$$
for all $V\in\calR$ and all $i=0,\ldots, n-1$,
we have $\calR\subseteq \calP_n$.  Thus,
$\phi^N(U)\subseteq\bigcup_{V\in \calP_n} V = A_n$,
proving the claim.

\textbf{Step 2}.
Next, we claim that
$\mu(\phi^{-N}(A_n)\smallsetminus A_N) \geq d^{-N} \mu(A_n)$
for any $n\geq 0$.
Since $A_n=\bigcup_{U\in\calP_n} U$ is a countable disjoint union,
it suffices to show that
$\mu(\phi^{-N}(U)\smallsetminus A_N) \geq d^{-N} \mu(U)$
for any $U\in\calP_n$.  Fix such a $U$.

By the hypotheses, each iterate $\phi^i(U_0)$ is a union
of elements of $\calP$, each of which in turn maps bijectively onto a
union of elements of $\calP$.  Thus, since $\phi^N(U_0)=\Jul_{\phi}$,
there is some $V_{N-1}\subseteq\phi^{N-1}(U_0)$ contained
in an element of $\calP$ such that $\phi$ maps $V_{N-1}$ bijectively
onto $U$.  By similar reasoning, there is some
$V_{N-2}\subseteq\phi^{N-2}(U_0)$ contained
in an element of $\calP$ such that $\phi$ maps $V_{N-2}$ bijectively
onto $V_{N-1}$.  Continuing in this fashion, there is some
$V_0\subseteq U_0$ such that $\phi^N$ maps $V_0$ bijectively onto $U$.
Since $N\geq 1$, we have
$$V_0\subseteq \phi^{-N}(U)\cap U_0
\subseteq \phi^{-N}(U)\smallsetminus A_N.$$
Thus, by properties of the Jacobian
discussed in Section~\ref{sec:invarmeas}, we have
\begin{align*}
\mu(U) &= \mu\big(\phi^N(V_0)\big)
= \int_{V_0} \Jac_{\phi^N}(x)\, d\mu(x)
= \int_{V_0} \frac{d^N}{\deg_x(\phi^N)}\, d\mu(x)
\\
& \leq \int_{V_0} d^N \, d\mu(x) = d^N \mu(V_0) \leq
d^N \mu\big(\phi^{-N}(U)\smallsetminus A_N\big),
\end{align*}
proving the claim.

\textbf{Step 3}.
By Steps~1 and~2 and the $\phi$-invariance of $\mu$, we have
\begin{align*}
\mu(A_{n+N}) & \leq \mu\big(\phi^{-N}(A_n) \cap A_N\big)
= \mu\big(\phi^{-N}(A_n)\big) -
\mu\big(\phi^{-N}(A_n) \smallsetminus A_N\big)
\\
& = \mu(A_n) - \mu\big(\phi^{-N}(A_n) \smallsetminus A_N\big)
\leq (1-d^{-N}) \mu(A_n)
\end{align*}
for any $n\geq 0$.
Thus,
\begin{equation}
\label{eq:sumAn}
\sum_{n\geq 1} \mu(A_n)
\leq \Big[ \sum_{m\geq 0} (1-d^{-N})^m \Big]\cdot
\Big[ \sum_{i=1}^N \mu(A_i) \Big] < \infty.
\end{equation}

\textbf{Step 4}.
For any $n\geq 0$, set $\calP'_n:=\calP_n\smallsetminus\calP_{n+1}$.
Applying similar reasoning as in Step~2, observe that
for any $n\geq 0$ and any $U\in\calP'_n$,
there is some $V\subseteq U$ such that
$\phi^n$ maps $V$ bijectively onto $U_0$.
Another computation with the integral of the Jacobian
then shows that $\mu(U_0)\leq d^n \mu(V) \leq d^n \mu(U)$.
Thus, $\mu(U) \geq d^{-n}\mu(U_0)$.

Moreover, any uncountable $U\in\calP$ must contain
a nonempty open set, forcing
$\phi^n(U)=\Jul_{\phi}\supseteq U_0$
for some $n\geq 0$.  Hence $U\in\calP'_n$.
Thus, $\calP':=\bigcup_{n\geq 0}\calP'_n$ consists of all
elements of $\calP$ except the countable ones,
which have measure zero.

We can now bound the entropy of $\calP$ using these observations:
\begin{align*}
h(\calP) &=
\sum_{U\in\calP} -\mu(U) \log(\mu(U))
= \sum_{n\geq 0} \sum_{U\in \calP'_n} - \mu(U) \log(\mu(U))
\\
&\leq
\sum_{n\geq 0} \sum_{U\in \calP'_n}
\mu(U) \big[ n\log d - \log(\mu(U_0)) \big]
\\
& = -\log \mu(U_0) \sum_{U\in\calP'} \mu(U)
+ \log d \sum_{n\geq 0} \sum_{U\in \calP'_n} n \mu(U)
\\
&= -\log \mu(U_0) + \log d \sum_{n\geq 1} \mu(A_n)
<\infty,
\end{align*}
where the final inequality is by Step~3.
\end{proof}

\begin{proof}[Proof of Theorem B]
\textbf{Step 1}.
Let $\calU$ be the set of finite intersections of the form
\begin{equation}
\label{eq:Uint}
U_{0} \cap \phi^{-1}(U_{1}) \cap \cdots \cap \phi^{-n}(U_{n}),
\end{equation}
where $n\geq 0$ and $U_0,\ldots,U_n\in\calP$.
We claim that every element of $\calU$ is path-connected,
proceeding by induction on $n$.  The sets with $n=0$ are
path-connected by hypothesis.  If the claim is known for sets with $n=m$,
an element of $\calU$
with $n=m+1$ is of the form $U \cap \phi^{-1}(V)$, where
$U\in\calP$ and $V$ is path-connected.
Since $\phi(U)$ and $V$ are path-connected, and since $\PBerk$
is uniquely path-connected, $\phi(U)\cap V$ is path-connected.
On the other hand, because $\phi$ is an open map
and, by hypothesis, is injective on $U$, the restriction
$\phi:U \cap \phi^{-1}(V) \to \phi(U)\cap V$ is a homeomorphism.
Thus, $U \cap \phi^{-1}(V)$ is path-connected, proving the claim.

\textbf{Step 2}.
By Lemma~\ref{lem:finent}, $\calP$ has finite entropy.
Recall from Section~\ref{sec:Julia}
that $\Jul_{\phi}$
has a topological basis of countably many open
connected affinoids intersected with $\Jul_{\phi}$.
Each open connected affinoid is, in turn, the intersection
of finitely many open disks.
Thus, to prove statement~(a), it suffices to show that
for any open disk $D\subseteq \PBerk$, the set $D\cap\Jul_{\phi}$
belongs to the $\sigma$-algebra generated by the set $\calU$ of Step~1.

Give such an open disk $D$, let $y\in\HH$ be its unique boundary point.
If $y\not\in\Jul_{\phi}$, then because $\Jul_{\phi}$ is connected,
$D\cap\Jul_{\phi}$ either is empty or is all of $\Jul_{\phi}$;
either way, it is in the $\sigma$-algebra generated by $\calU$.
Thus, we may assume that $y\in\Jul_{\phi}$.

For any $x\in D\cap\Jul_{\phi}$, note that $x\neq y$, since
$y\not\in D$.  We claim that for some large enough
$n\geq 0$, $\phi^n(x)$ and $\phi^n(y)$ lie in distinct
partition elements $U\in\calP$.
Otherwise, by the hypotheses that some iterate of $\phi$ is expanding
on each $U$, and that $\Jul_{\phi}$ has finite $d_{\HH}$-diameter,
there is some $m\geq 0$ such that
$d_{\HH}(\phi^m(x),\phi^m(y)) > \diam_{\HH}(\Jul_{\phi})$,
contradicting the fact that $\phi^m(x),\phi^m(y)\in\Jul_{\phi}$.

Hence, there is some $V_x\in \calU$ such that $x\in V_x$
but $y \not\in V_x$.  Since $V_x$ is connected by the claim
of Step~1, we must have $x\in V_x\subseteq D\cap\Jul_{\phi}$.  Thus,
$D\cap\Jul_{\phi} = \bigcup_{x\in D\cap\Jul_{\phi}} V_x$.
However, since $\calU$ is countable, the union on the right is
in fact a countable union.  Therefore, $D\cap\Jul_{\phi}$
belongs to the $\sigma$-algebra on $\Jul_{\phi}$
generated by $\calU$, proving statement~(a).

\textbf{Step 3}.
Our proof of statement~(b) consists mainly of verifying the
hypotheses of an entropy formula of Rokhlin.
First, as we noted in Section~\ref{sec:Julia}, $\Jul_{\phi}$ is a separable
metrizable space.  As a closed subset of $\PBerk$, it is also compact.
Thus, by \cite[Theorem~15.4.10]{Roy}, it is a Lebesgue space.
(This is a technical condition meaning
that as a measure space, $\Jul_{\phi}$
is isomorphic to $[0,1]$ plus perhaps countably many points.
See also \cite[Theorem~2.1]{Wal}.)
Second, recall from Section~\ref{sec:invarmeas}
that $\mu$ is ergodic on $\Jul_{\phi}$.
Third, $\calP$ is a countable one-sided generator for
$\phi:\Jul_{\phi}\to\Jul_{\phi}$, and of finite entropy,
by statement~(a).
Fourth, $\phi$ is essentially countable-to-one
in the sense of \cite[Definition~2.9.2]{PU}, since
$\phi$ is finite-to-one.

Thus, by Rokhlin's formula
(see, for example, \cite[Theorem~2.9.7]{PU}),
the hypotheses of which we checked in the previous paragraph,
and by equation~\eqref{eq:jacfmla}, we have
$$h_{\mu}(\phi) = \int_{\PBerk} \log \Jac_{\phi}(x) \, d\mu(x)
=\int_{\PBerk} \log\Big( \frac{d}{\deg_x(\phi)} \Big) \, d\mu(x),$$
proving statement~(b).

\textbf{Step 4}.
Let $Y$ denote the space of one-sided symbol sequences described in
part~(c) of Theorem~B: the symbol space is the set $\calA$
of uncountable elements of $\calP$, and symbol
$U\in\calA$ can be followed by symbol $V\in\calA$
if and only if $V\in S_U$.
The topology on $Y$ is inherited from
the product topology on $\calA^{\NN}$, where we
equip the symbol space $\calA$ with the discrete topology.
Let $T:Y\to Y$ denote the shift map on $Y$.

Writing each $U\in\calP$ as the union of an open set
and a countable set $C_U$, let
$Z_0\subseteq\Jul_{\phi}$ be the (countable) union of all of the
countable sets $C_U$.
Then, let 
$Z:=\bigcup_{n\geq 0}\phi^{-n}(Z_0)\subseteq\Jul_{\phi}$ 
be the (countable) set of all points with iterates in $Z_0$.
Let $J:=\Jul_{\phi}\smallsetminus Z$.


Define $Q: J \to Y$
by letting $Q(x)$ be the sequence
$\{U_n\}_{n\geq 0}$, where $U_n$ is the unique element of $\calA$
containing $\phi^n(x)$.  Clearly $Q\circ\phi = T\circ Q$.
Moreover, since $J\subseteq \PBerk$ is Hausdorff,
part~(a) implies that different points in $J$
have different symbol sequences; that is, $Q$ is injective.

\textbf{Step 5}.
Let $W:=Y\smallsetminus Q(J)$.
We claim that $W$ is a countable set.

Given $u=\{U_i\}_{i\in\NN}\in W$, let $E_n\subseteq U_0$ be the
finite intersection given by expression~\eqref{eq:Uint}, for each $n\geq 0$.
Clearly $U_0=E_0\supseteq E_1 \supseteq E_2 \supseteq\cdots$.
Define $B_u$ to be the set of all points in $\PBerk$ that
are accumulation points of sequences
$\{x_n\}_{n\geq 0}\subseteq U_0$ with $x_n\in E_n$.

For each $i\in\NN$, the set
$\tilde{U}_i:=\Jul_{\phi}\smallsetminus (Z_0\cup U_i)$
is a union of opens and hence is open in $\Jul_{\phi}$.
Thus, $\phi^{-i}(\tilde{U}_i)$ is also open.  In addition,
for any sequence $\{x_n\}_{n\geq 0}$ as in the previous paragraph,
we have $x_n\not\in \phi^{-i}(\tilde{U}_i)$ for all $n\geq i$.
Therefore, $B_u\cap \phi^{-i}(\tilde{U}_i)=\varnothing$,
and hence $B_u\subseteq \phi^{-i}(Z_0\cup U_i)$ for all $i\in\NN$.
If there were some $x\in B_u \cap \bigcap_{i\in\NN} \phi^{-i}(U_i)$,
then we would have $u=Q(x)$, contradicting our assumption
that $u\in W$.  Instead, then, we have
$B_u\subseteq \bigcup_{i\in\NN}\phi^{-i}(Z_0)=Z$.


We have shown that $B_u$ is a subset of $Z$ for any $u\in W$.
For the remainder of this step, fix $z\in Z$,
and let $u=\{U_i\}_{i\in\NN}$ and $v=\{V_i\}_{i\in\NN}$
be two distinct symbol sequences in $W$ with $z\in B_u\cap B_v$.
Choose $M\geq 0$ large enough that $\phi^M(z)\in Z_0$,
choose $N\geq 0$ large enough that $U_N\neq V_N$, and
let $n=\max\{M,N\}$.
Define $E_n:=\bigcap_{i=0}^n \phi^{-i}(U_i)$
and $F_n:=\bigcap_{i=0}^n \phi^{-i}(V_i)$, both of which
are connected, by Step~1.  In addition, by our choice of $n$,
we have $E_n\cap F_n=\varnothing$;
thus, $E_n$ and $F_n$ cannot both intersect a common
residue class at $z$.
On the other hand,
although one or the other might contain $z$ itself,
both $E_n$ and $F_n$ contain points other than $z$.

Recall from the end of Section~\ref{sec:Julia}, however,
that only countably many residue classes at $z$ can intersect
$\Jul_{\phi}$.  Since $E_n$ and $F_n$ are subsets of
$\Jul_{\phi}$, we have shown that there can only be countably
many $u\in W$ with $z\in B_u$.  Because $Z$ itself is countable,
we have proven our claim that $W$ is countable.

\textbf{Step 6}.
Define $Y':= Y\smallsetminus W = Q(J)$, so that $Q:J\to Y'$ is bijective.
$Y'$ has a topological basis consisting of cylinder sets, i.e.,
sets given by specifying the first $n$ symbols
$U_0,\ldots,U_n\in\calA$ of the symbol sequence $u\in Y'$.
The inverse image of such a cylinder set under $Q$ is given
by expression~\eqref{eq:Uint} and hence is measurable in $J$.
Thus $Q:J\to Y'$ is (Borel) measurable.
Conversely, by part~(a), any Borel subset $U$ of $Y'$
belongs to the $\sigma$-algebra generated by $\calU$,
and hence $Q(U)$ belongs to the $\sigma$-algebra generated
by the cylinder sets on $Y'$, which is the Borel $\sigma$-algebra
on $Y'$.  Thus, the bijection $Q:J\to Y'$ gives an isomorphism of
(Borel) $\sigma$-algebras.

Recalling from Step~4 that $Q\circ\phi = T\circ Q$, then,
there is a one-to-one correspondence between 
the set $M(J,\phi)$ of $\phi$-invariant
Borel probability measures on $J$ and
the set $M(Y',T)$ of $T$-invariant
Borel probability measures on $Y'$
Moreover, if $\mu\in M(J,\phi)$ and $\nu\in M(Y',T)$
are corresponding such measures,
then the entropies $h_{\mu}(J,\phi)$ and $h_{\nu}(Y',T)$ clearly coincide.

Meanwhile, since $Z=\Jul_{\phi}\smallsetminus J$ is $\phi$-invariant,
any measure $\mu\in M(J,\phi)$ extends to a measure
in $M(\Jul_{\phi},\phi)$ via $\mu(Z)=0$.
Conversely, since $Z$ is countable,
any measure $\mu\in M(\Jul_{\phi},\phi)$
either has $\mu(Z)=1$ and hence has entropy zero,
or else it induces a measure $\tilde{\mu}\in M(J,\phi)$,
given by $\tilde{\mu}(E)=\mu(E\cap J)/\mu(J)$, of
greater entropy.
A similar relationship applies to $M(Y,T)$ and $M(Y',T)$,
since $W$ is also countable and $T$-invariant.

Because $\Jul_{\phi}$ is a compact metrizable space,
the variational principle applies to it.
Meanwhile, Gurevich proved \cite{Gur70} (see also \cite[Theorem~1.3]{Rue})
that $h_{\textup{Gur}}(T)$
is the supremum of all the
measure-theoretic entropies of $T:Y\to Y$.
Thus,
\begin{align*}
h_{\textup{top}}(\phi)
&= \sup\{h_{\mu}(\Jul_{\phi}, \phi) : \mu\in M(\Jul_{\phi},\phi) \}
= \sup\{h_{\mu}(J, \phi) : \mu\in M(J,\phi) \}
\\
&= \sup\{h_{\nu}(Y',T) : \nu\in M(Y',T) \}
= \sup\{h_{\nu}(Y,T) : \nu\in M(Y,T) \}
=h_{\textup{Gur}}(T).
\qedhere
\end{align*}
\end{proof}

\section{A degree 6 example}
\label{sec:examples}

\subsection{With arbitrary residue field}
\label{ssec:sextic}
Fix $a\in\Cv^{\times}$ with $0<|a|_v<1$.  Fix $b\in\Cv$ with
$|b|_v=|b-1|_v=1$,
and define
$$\phi(z) = \frac{az^6 + 1}{az^6 + z(z-1)(z-b)}
= 1 + \frac{1-z(z-1)(z-b)}{az^6 + z(z-1)(z-b)}.$$
Let $\Gamma\subseteq \PBerk$ be the convex hull
of the four points
$$\zeta(0,|a|_v^{-1/2}),
\zeta(0,|a|_v^{1/2}),
\zeta(1,|a|_v^{1/2}),
\text{ and }
\zeta(b,|a|_v^{1/2}).$$
Thus, $\Gamma$ is a tree consisting of four segments,
which we denote $I$, $J_0$, $J_1$, and $J_b$,
extending from the Gauss point $\zeta(0,1)$ to the
above four points, respectively.  See Figure~\ref{fig:sextictree}.

\begin{figure}
\label{fig:sextictree}
\includegraphics{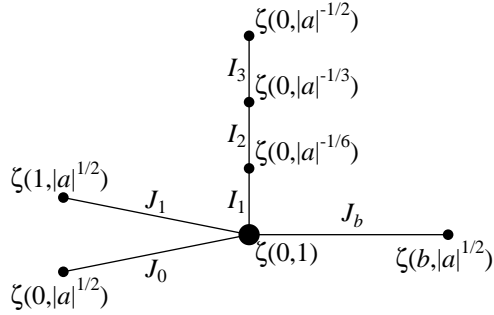}
\caption{The tree $\Gamma$ for the map of Section~\ref{ssec:sextic}.}
\end{figure}

Break $I$ into the following three equal-length pieces:
$$I_1=[\zeta(0,1),\zeta(0,|a|_v^{-1/6})], \;
I_2=[\zeta(0,|a|_v^{-1/6}),\zeta(0,|a|_v^{-1/3})],$$
$$\text{and} \quad I_3=[\zeta(0,|a|_v^{-1/3}),\zeta(0,|a|_v^{-1/2})].$$
Note that the Gauss point $x_0=\zeta(0,1)\in I$ is fixed and repelling,
since the reduction
$\overline{\phi}(z) = 1/(z(z-1)(z-\overline{b}))$ has degree $3>1$.
In particular, $x_0\in\Jul_{\phi}$.
Meanwhile, recalling the definition of the symbol
$\approx$ from \eqref{eq:approxdef}, observe that:
\begin{itemize}
\item
$\phi(\zeta(0,r))=\zeta(0,1/r^3)$ for $1<r<|a|_v^{-1/6}$,
since $\phi(z)\approx 1/z^3$ for $1<|z|_v<|a|_v^{-1/6}$.  Thus, $\phi$
maps $I_1$ bijectively onto  $J_0$,
stretching $d_{\HH}$ by a factor of $3$.
\item
$\phi(\zeta(0,r))=\zeta(0,|a|_vr^3)$
for $|a|_v^{-1/6} < r < |a|_v^{-1/3}$,
since $\phi(z)\approx az^3$ for $|a|_v^{-1/6}<|z|<|a|_v^{-1/3}$.
Thus, $\phi$
maps $I_2$ bijectively onto $J_0$,
again stretching $d_{\HH}$ by a factor of $3$,
but this time oriented in the opposite direction.
\item
$\phi(\zeta(0,r))=\zeta(1,|a|_v^{-1} r^{-3})$
for $|a|_v^{-1/3} < r < |a|_v^{-1/2}$,
since $\phi(z)-1\approx -1/(az^3)$
for $|a|_v^{-1/3}<|z|_v<|a|_v^{-1/2}$.  Thus, $\phi$
maps $I_3$ bijectively onto $J_1$,
stretching $d_{\HH}$ by a factor of $3$.
\end{itemize}
In addition, $\phi$ maps each of $J_0$, $J_1$, and $J_b$ bijectively,
and in fact isometrically, onto $I$.

Therefore, $\phi^2$ maps each of $I_1$, $I_2$, and $I_3$ onto $I$,
stretching distances on each by a factor of $3$.  In addition,
$\Gamma\subseteq I\cup \phi^{-1}(I)$, and $\phi$ maps $\zeta(0,|a|_v^{-1/3})$
to the Gauss point $x_0=\zeta(0,1)$ with local degree $3$.
Hence, $\phi^{-1}(x_0)=\{x_0, \zeta(0,|a|_v^{-1/3})\}\subseteq \Gamma$,
and thus $x_0\in I\subseteq \Gamma$ satisfies the hypotheses of Theorem~A
for $\phi$.
As a result, the Julia set $\Jul_{\phi}$ is connected
and has a dense set of branch points.
In fact, since
$\overline{\phi}(z) = 1/(z(z-1)(z-\overline{b}))$ is separable,
those branch points have infinite branching, by Theorem~\ref{thm:infbranch}.

\subsection{Intervals of measure zero}
\label{ssec:leaves6}

Favre and Rivera-Letelier \cite{FRL3} have announced a proof
of the following fact: if the invariant measure $\mu_{\psi}$ of a rational
function $\psi\in\Cv(z)$ charges an interval in $\HH$, then the
Julia set $\Jul_{\psi}$ is also contained in an interval.
In particular, the interval $I$ contained in our
Julia set $\Jul_{\phi}$ of Section~\ref{ssec:sextic}
must have mass $\mu_{\phi}(I)=0$.

We can also see this identity directly for our map $\phi$, as follows.
As we noted, $\phi^2$ maps each of $I_1$,
$I_2$, and $I_3$ bijectively onto $I$, with $\deg_x(\phi^2)=3$
for all $x\in I$ other than the (finitely many) endpoints of
$I_1$, $I_2$, and $I_3$.
By the Jacobian formula~\eqref{eq:jacfmla}, for each $i=1,2,3$,
we have
$$\mu_{\phi}(I) = \mu_{\phi}(\phi^2(I_i))
=\int_{I_i} \Jac_{\phi^2}(x) \, d\mu_{\phi}(x)
=\int_{I_i} \frac{6^2}{3} \, d\mu_{\phi}(x)
= 12\mu_{\phi}(I_i).$$
Thus, since $I$ is the union of $I_1$, $I_2$, and $I_3$, any two
of which intersect on a set of measure zero, we have
$$\mu_{\phi}(I)
= \mu_{\phi}(I_1) + \mu_{\phi}(I_2) + \mu_{\phi}(I_3) 
= \frac{1}{4}\mu_{\phi}(I),$$
proving our claim that $\mu_{\phi}(I)=0$.

It follows immediately from this observation
that $\mu_{\phi}(\bigcup_{n\geq 0} \phi^{-n}(I)) = 0$.
As noted in the proof of Theorem~A, the Julia set $\Jul_{\phi}$
is the closure of $\bigcup_{n\geq 0} \phi^{-n}(I)$. 
However, viewing $\Jul_{\phi}$ as a tree with infinite branching,
this set $\bigcup_{n\geq 0} \phi^{-n}(I)$ includes
all of the interior points of the tree. Thus, the set $L$ of
leaves --- that is, $L$ consists of those $x\in\Jul_{\phi}$
for which $\Jul_{\phi}\smallsetminus\{x\}$ is still connected
--- has mass $\mu_{\phi}(L)=1$.

\subsection{The entropy in residue characteristic 3}
\label{ssec:ent6}

We will now compute the measure-theoretic and topological entropies
of the sextic map $\phi$ of Section~\ref{ssec:sextic}
when the residue characteristic $\charact k$ is $3$,
when $|3|_v\leq |a|_v < 1$, and when neither of
$1,\bar{b}\in \PP^1(k)$ is postcritical
under the action of $\bar{\phi}(z)=1/(z(z-1)(z-\bar{b}))$.
Since the only critical points of $\bar{\phi}$ are $\infty$ 
(with forward orbit $\{0,\infty\}$) and
$\gamma:=-\bar{b}/(\bar{b}+1)$, this last condition says simply that
$1$ and $\bar{b}$ are not in the forward
orbit of $\gamma$. In particular, the condition holds for
the map given in the introduction, which
has $b=-1$, and hence $\gamma=\infty$.

\begin{claim}
\label{claim:fatou}
Let $y:=\zeta(0,|a|_v^{-1/6})$. Then
$$\phi(y) = \zeta(0,|a|_v^{1/2}), \quad
\phi^2(y) = \phi^{4}(y)=\zeta(0,|a|_v^{-1/2}), \quad\text{and}\quad
\phi^3(y) = \zeta(1,|a|_v^{1/2}),$$
with $\deg_y(\phi)=6$ and $\deg_{\phi^2(y)}(\phi^2)=3$.
Moreover, setting
$$W:=\{x\in\PBerk : |x|_v=|a|_v^{-1/6}\}\smallsetminus\{y\},$$
we have $\phi(W)=\DbarBerk(0,|a|_v^{1/2})\smallsetminus\{\phi(y)\}$,
and
$$\phi^2(W)=\phi^4(W)=
\{x\in\PBerk : |x|_v \geq |a|_v^{-1/2} \} \smallsetminus
\{\zeta(0,|a|_v^{-1/2})\}.$$
In particular, $\phi^i(W)\subseteq\Fat_{\phi}$ for all $i\geq 0$.
\end{claim}

\begin{proof}
Let $\psi(z) := a^{-1/2}\phi(a^{-1/6}z)$, so that
$\bar{\psi}(z)=(z^6+1)/z^3$.
Thus, $\deg\bar{\psi}=6$, and hence
we have $\phi(y) = \zeta(0,|a|_v^{1/2})$, with $\deg_y(\phi)=6$.
Similar computations confirm the claimed values of $\phi^i(y)$
for $i=2,3,4$, and that $\deg_{\phi^2(y)}(\phi^2)=3$.

The characterization of each $\phi^i(W)$ 
also comes from direct computation.
Since $\phi^j(W)=\phi^2(W)$ for $j\geq 2$ even, and $\phi^j(W)=\phi^3(W)$
for $j\geq 3$ odd, the union of the forward images of any $\phi^i(W)$
omits infinitely many points of $\PBerk$, and hence
$\phi^i(W)\subseteq\Fat_{\phi}$.
\end{proof}

Define $V_0:=\{\phi(y)\} =\{\zeta(0,|a|_v^{1/2})\}$,
$V_\infty:=\{\phi^2(y)\} =\{\zeta(0,|a|_v^{-1/2})\}$,
and $V_1:=\{\phi^3(y)\} =\{\zeta(1,|a|_v^{1/2})\}$.
By Claim~\ref{claim:fatou}, each of $V_0$, $V_1$, $V_\infty$
is an endpoint of $\calJ_{\phi}$.

The set
$\Jul_{\phi}\smallsetminus\{\zeta(0,1)\}$ consists of countably
many branches: one extending towards $\infty$, and the
rest of the form $U_c:=\Jul_{\phi}\cap \DBerk(c,1)$ where $c\in k$.
(Here, we abuse notation by lifting $c\in k$ to an element of $\ints$,
which we also denote $c$.)
Let $\calC\subseteq k$ be the set of those $c\in k$
for which $U_c\neq\varnothing$,
and let $\calC':=\calC\smallsetminus\{0,1\}$.
Define $U'_0:=U_0\smallsetminus V_0$ and $U'_1:=U_1\smallsetminus V_1$.
Meanwhile, let
$$
U_{\infty,1} :=
\{x\in\Jul_{\phi}:|x|_v \geq |a|_v^{-1/6}\}\smallsetminus V_\infty,
\quad\text{and}\quad
U_{\infty,2} := \{x\in\Jul_{\phi}: 1< |x|_v < |a|_v^{-1/6}\}.$$
Finally, let $V=\{\zeta(0,1)\}$.
Clearly,
$$\calP:=\{V,V_0,V_1,V_\infty,
U_{\infty,1},U_{\infty,2},U'_0,U'_1\}\cup\{U_c:c\in\calC'\}$$
is a countable partition of $\Jul_{\phi}$,
where $V$, $V_0$, $V_1$, and $V_\infty$ are
singletons, $U_{\infty,1}$ is the union
of an open set and the singleton $\{y\}$,
and the remaining elements of $\calP$ are open in $\Jul_{\phi}$.

We have $\phi(V)=V$, $\phi(V_0)=\phi(V_1)=V_{\infty}$, and 
$\phi(V_\infty)=V_1$.
In the notation of Section~\ref{ssec:sextic},
$U_{\infty,2}\cap \Gamma$ is the interval $I_1$
with its endpoints removed, while
$U_{\infty,1}\cap \Gamma=I_2\cup I_3$.
Thus,
$\phi(U_{\infty,2})=U'_0$,
$\phi(U_{\infty,1})=\Jul_{\phi}$, and
$$\phi(U'_0)=\phi(U'_1)= U_{\infty,1}\cup U_{\infty,2},
\quad\text{while}\quad
\phi(U_{\bar{b}}) = U_{\infty,1}\cup U_{\infty,2}\cup V_{\infty}.$$
There are also three values of $c$ (namely, the roots of
$\bar{\phi}(c)=1$) for which $\phi(U_c)=U'_1\cup V_1$.
Finally, $\phi(U_c) = U_{\bar{\phi}(c)}$ for every other
$c\in\calC'\smallsetminus\{\bar{b}\}$.
Note, by our assumed condition on $b$, that $\phi$ is injective
on $U'_0$, $U'_1$, and each $U_c$,
with local degree $1$ at all points of these sets.


\begin{claim}
\label{claim:temp}
With notation as above, $\deg_x(\phi)=3$ for all
$x\in U_{\infty,1}\cup U_{\infty,2}\smallsetminus\{y\}$.
Moreover, $\Jul_{\phi}$ has finite hyperbolic diameter,
and $\phi$ is injective on both
$U_{\infty,1}$ and $U_{\infty,2}$.
\end{claim}

\begin{proof}
For $\alpha, w \in\Cv$ with
$1< |w|_v < |\alpha|_v < |a|_v^{-1/6}$, we have
$\phi(\alpha+w) - \alpha^{-3} \approx -\alpha^{-6} w^3$,
and thus $\phi:\Dbar(\alpha,r) \to \Dbar(\alpha^{-3}, |\alpha|_v^{-6} r^3)$
is $3$-to-$1$ for all $1\leq r\leq |\alpha|_v$.
(Here, we are using the hypothesis that $|3|_v\leq |a|_v$.)
Hence, $\deg_{\zeta(\alpha,r)}(\phi)=3$ for all such $\alpha,r$.
In addition, we have
$\phi(\Dbar(\alpha,1))=\Dbar(\alpha^{-3},|\alpha|_v^{-6})$.
Since $\phi(z)\approx 1/bz$ for $z\in D(0,1)$,
it follows that $\phi^2:\Dbar(\alpha,1) \to \Dbar(b^{-1}\alpha^3, 1)$.

Similarly, if $|a|_v^{-1/6} < |\alpha|_v < |a|_v^{-1/3}$
and $1\leq r\leq |\alpha|_v$, then
$\phi:\Dbar(\alpha,r)\to\Dbar(a\alpha^3, |a|_v r^3)$
is also $3$-to-$1$, with
$\phi(\Dbar(\alpha,1)) = \Dbar(a\alpha^3, |a|_v)$.
Likewise, if $|a|_v^{-1/3} < |\alpha|_v < |a|_v^{-1/2}$
and $1\leq r \leq |\alpha|_v$, then
$\phi:\Dbar(\alpha,r)\to\Dbar(1/(1+a\alpha^3), |a\alpha^6|_v^{-1}r^3)$
is again $3$-to-$1$, with
$\phi(\Dbar(\alpha,1))= \Dbar((a(b-1)\alpha^3)^{-1}, |a|_v)$.
In either case, we get
$\deg_{\zeta(\alpha,r)}(\phi)=3$ for all such $\alpha,r$,
and $\phi^2(\Dbar(\alpha,1))\subseteq \Dbar(\phi^2(\beta,1))$
with $|\beta|_v>1$.

If $1\leq r \leq|\alpha|_v = |a|_v^{-1/3}$,
then $\deg_{\zeta(\alpha,r)}(\phi)=3$,
although the point $\phi(\zeta(\alpha,r))$ could lie in $\DbarBerk(0,1)$.
In fact, the image $\phi(D)$ of the disk $D=\Dbar(\alpha,1)$ is of the form
$\Dbar(\beta,|a\beta^2|_v)$ with $|\beta|_v\geq 1$.
If $|\beta|_v\leq |a|_v^{-1/2}$, so that $\phi(D)$
is not contained in $\phi^2(W)\subseteq\Fat_{\phi}$,
then the radius of $\phi(D)$ is at most $1$.
In addition, if $|\beta|_v=1$, then $\phi(D)=\Dbar(\beta,|a|_v)$.

One consequence of the previous three paragraphs is that for any disk
$D=\DbarBerk(\alpha,1)$ with $1<|\alpha|_v < |a|^{-1/2}$, 
then either $D\subseteq W$, or 
\begin{itemize}
\item $\phi(D)$ is of the form
$\DbarBerk(\beta,s)$, where $|\beta|_v\leq 1$ and $s\leq |a|_v$, or
\item $\phi(D)$ or $\phi^2(D)$ is of the form
$\DbarBerk(\beta,s)$, where $|\beta|_v> 1$ and $s\leq 1$.
\end{itemize}
Similarly, if $D=\DbarBerk(\alpha,|a|_v)$ with $|\alpha|_v\leq 1$,
we get the same conclusions.
All such disks therefore map into one another, and hence they all
lie in the Fatou set.  Thus,
$\Jul_{\phi}\smallsetminus \DbarBerk(0,1)$
consists only of points of the form $x=\zeta(\alpha,r)$ with
$r\geq 1$.  By the same three paragraphs,
$\deg_{x}(\phi)=3$ for all
$x\in U_{\infty,1}\cup U_{\infty,2}\smallsetminus\{y\}$,
proving the first statement of the claim.

Recall that $U_{\infty,1}\smallsetminus\{y\}$
maps onto $\Jul_{\phi}\smallsetminus\{\phi(y)\}$,
and $U_{\infty,2}$ maps onto $U'_0$,
while each $U_c$ maps bijectively onto $U_{\bar{\phi}(c)}$.
(Or onto $U'_1 \cup V_1$, if $\bar{\phi}(c)=1$.)
Thus, by equation~\eqref{eq:degsum},
$\phi$ must be injective on both $U_{\infty,1}$ and $U_{\infty,2}$.
Finally, we have seen that if $\zeta(\alpha,r)\in\Jul_{\phi}$,
then either $1<|\alpha|_v\leq |a|_v^{-1/2}$ and $1\leq r\leq |\alpha|_v$,
or else $|a|_v^{1/2}\leq |\alpha|_v \leq 1$ and $|a|_v \leq r \leq |\alpha|_v$.
Any such point lies within hyperbolic distance less than $-2\log|a|_v$
from $\zeta(0,1)$.  Since $\Jul_{\phi}$ is connected, it follows that
any point of $\Jul_{\phi}$ (not just those of type II or III)
also lies within that distance of $\zeta(0,1)$.
Thus, $\Jul_{\phi}$ has finite hyperbolic diameter,
proving the claim.
\end{proof}

Claim~\ref{claim:temp} confirms most of the hypotheses
of Theorem~B, including the first bullet point.
In addition, because $y$ is the only point in
$\{x\in \Jul_{\phi} : |x|_v = |a|_v^{-1/6}\}$, the
sets $U_{\infty,1}\smallsetminus\{y\}$ and $U_{\infty,2}$
are both connected, and $\deg_x(\phi)=3$ is constant on each.
Thus, $d_{\HH}(\phi(x),\phi(x')) = 3d_{\HH}(x,x')$
for any $x,x'\in U_{\infty,1}$, and similarly for $U_{\infty,2}$.
Since $U'_0$ and $U'_1$ map onto $U_{\infty,1}\cup U_{\infty,2}$,
while $U_{\bar{b}}$ maps onto $U_{\infty,1}\cup U_{\infty,2} \cup V_\infty$,
and every other $U_c$ maps onto $U_{\bar{\phi}(c)}$
(and hence eventually onto $U_{\infty,1}\cup U_{\infty,2} \cup V_\infty$),
we have verified the remaining hypotheses of Theorem~B.

Invoking part~(b) of the Theorem, then,
$$h_{\mu}(\phi) = \int_{\PBerk} \log(6/\deg_x(\phi)) \, d\mu(x)
=(\log 6) \mu(\DbarBerk(0,1))
+ (\log 2) \mu(U_{\infty,1}\cup U_{\infty,2}),$$
where $\mu=\mu_{\phi}$ is the invariant measure.
However, since $\phi:U_{\infty,1}\to\Jul_{\phi}$ is bijective,
with constant local degree $\deg_{x}(\phi)=3$
except at the one point $x=y$,
the defining property of $\Jac_{\phi}$ shows that
$\mu(U_{\infty,1}) = 1/2$.  Similarly,
$$\mu(U'_0) = \frac{1}{6}\mu(U_{\infty,1}\cup U_{\infty,2})
= \frac{1}{12} + \frac{1}{6}\mu(U_{\infty,2}),
\quad\text{and}\quad
\mu(U_{\infty,2}) = \frac{1}{2}\mu(U'_0).$$
Thus, $\mu(U_{\infty,2})=1/22$, and hence
$\mu(U_{\infty,1}\cup U_{\infty,2}) = 6/11$.
Substituting into the above formula therefore gives
$$h_{\mu}(\phi) = \frac{5}{11}\log 6 + \frac{6}{11}\log 2
= \log 2 + \frac{5}{11}\log 3.$$

We now compute the Gurevich entropy of the resulting
Markov shift on the set of symbols
$\calA = \{U_{\infty,1}, U_{\infty,2}, U'_0, U'_1, U_{\bar{b}}, \ldots\}$.
With that ordering of the symbols, the transition matrix
is the transpose of
$$
\begin{bmatrix}
1 & 0 & 1 & 1 & 1 & 0 & 0 & 0 & 0 & 0 & 0 & 0 & 0 & 0 & 0 & 0 & 0 & \cdots
\\
1 & 0 & 1 & 1 & 1 & 0 & 0 & 0 & 0 & 0 & 0 & 0 & 0 & 0 & 0 & 0 & 0 & \cdots
\\
1 & 1 & 0 & 0 & 0 & 0 & 0 & 0 & 0 & 0 & 0 & 0 & 0 & 0 & 0 & 0 & 0 & \cdots
\\
1 & 0 & 0 & 0 & 0 & 1 & 1 & 1 & 0 & 0 & 0 & 0 & 0 & 0 & 0 & 0 & 0 & \cdots
\\
1 & 0 & 0 & 0 & 0 & 0 & 0 & 0 & 1 & 1 & 1 & 0 & 0 & 0 & 0 & 0 & 0 & \cdots
\\
1 & 0 & 0 & 0 & 0 & 0 & 0 & 0 & 0 & 0 & 0 & 1 & 1 & 1 & 0 & 0 & 0 & \cdots
\\
1 & 0 & 0 & 0 & 0 & 0 & 0 & 0 & 0 & 0 & 0 & 0 & 0 & 0 & 1 & 1 & 1 & \cdots
\\
\vdots & \vdots & \vdots & \vdots & \vdots &
\vdots & \vdots & \vdots & \vdots & \vdots &
\vdots & \vdots & \vdots & \vdots & \vdots &
\vdots & \vdots & \ddots
\end{bmatrix}
$$
For notational convenience, we will identify each
state in $\calA$ with the number of its corresponding column;
that is, $U_{\infty,1}$ is $1$, $U_{\infty,2}$ is $2$,
$U'_{0}$ is $3$, $U'_{1}$ is $4$, $U_{\bar{b}}$ is $5$, etc.
We will compute $F_1(z)$, the generating function
for first-return loops at state~1.
As described in Section~\ref{sec:gurevich}, we can
compute $h_{\textup{Gur}}(T)$, and hence
$h_{\textup{top}}(\phi)$, from the roots of $1-F_1$.

To count first-return loops at $1$ of various lengths, let us consider
four separate types of such loops.
\begin{enumerate}
\item
The unique length one loop $1,1$,
giving a contribution of $z$ to $F_1(z)$.

\item
Loops beginning $1,2,\ldots$. The next state after $2$
must be $3$, followed by either $1$ or $2$. Thus,
we get one first-return loop of every
odd length at least $3$, of the form $1,2,3,2,3,\ldots,3,1$.
The type~(b) loops therefore contribute
$z^3 + z^5 + z^7 + \cdots = z^3/(1-z^2)$ to $F_1(z)$.

\item
Loops beginning $1,a$, where $a\in\{3,4,5\}$.
The next state must be either $1$ or $2$. Thus,
for each of $a=3,4,5$,
we get one first-return loop of every
positive even length, of the form $1,a,2,3,2,3,\ldots,3,1$
(where there could be zero copies of $2,3$).
The type~(c) loops therefore contribute
$3(z^2 + z^4 + z^6 + \cdots) = 3z^2/(1-z^2)$ to $F_1(z)$.

\item
Loops beginning $1,a$, where $a\geq 6$.
Any such state $a$ is followed by a unique path
of some length $k$ through states numbered $6$ and higher
to either state $4$ or $5$.
Call such a state $a$ a $k$-state. For each $k\geq 1$, there are $3^k$
$k$-states with paths to $4$, and $3^k$ with paths to $5$,
for a total of $2\cdot 3^k$ $k$-states.

After $k$ steps, each loop starting from a $k$-state
looks like the tail of one of type~(c), giving a loop of
length $k+2m$ for each $m\geq 1$ and each $k$-state.
Together, then, all $2\cdot 3^k$ $k$-states contribute
$2\cdot 3^k z^k (z^2 + z^4 + z^6 + \cdots)
= 2\cdot 3^k z^{k+2}/(1-z^2)$ to $F_1(z)$.
Summing across all $k\geq 1$, then, the type~(d) loops contribute

$\dsps \frac{6 z^3}{(1-z^2)} (1 + 3z + (3z)^2 + \cdots)
= \frac{6 z^3}{(1-z^2)(1-3z)}$ to $F_1(z)$.
\end{enumerate}

Adding up all four types' contributions, we have
$$F_1(z) = z + \frac{z^3 + 3z^2}{(1-z^2)} +
\frac{6 z^3}{(1-z^2)(1-3z)}
=\frac{z-3z^3}{(1-z^2)(1-3z)}.$$
Thus,
$$1- F_1(z)=\frac{(1-3z - z^2 + 3z^3) - (z-3z^3)}{(1-z^2)(1-3z)}
= \frac{1-4z - z^2 + 6z^3}{(1-z^2)(1-3z)}.$$
The root of $F_1(z)$ of smallest absolute value is $r=1/\lambda$,
where
$\lambda\approx 3.8558\ldots$ is the largest (real) root of the polynomial
$t^3 - 4t^2 - t + 6$.
Hence, as claimed in the introduction,
$$h_{\mu}(\phi) = \log 2 + \frac{5}{11} \log 3
\quad\text{and}\quad
h_{\textup{top}}(\phi) = \log \lambda.$$

\begin{remark}
\label{rem:sextic}
The assumption that  the residue characteristic $p=\charact k$ of $\Cv$
is $3$ is essential to the computations of this section.
Indeed, if $p\neq 3$, then
$\deg_{x}(\phi)=1$ for all $x=\zeta(c,r)$ with
$1,r<|c|_v$, except possibly for $|c|_v=|a|_v^{-1/6}$.
As a result, $\phi$ is injective and preserves the hyperbolic distance
on all of the side branches off $I\smallsetminus\{\zeta(0,|a|_v^{-1/6})\}$.
Thus, outside of a certain measure zero set (consisting of
those $x\in\Jul_{\phi}$
for which all but finitely many iterates $\phi^n(x)$ lie in branches
off $|a|_v^{-1/6}$), the points of the leaf set $L$ of
Section~\ref{ssec:leaves6} lie at infinite
hyperbolic distance from $\zeta(0,1)$. Hence,
$\mu_{\phi}$-almost all points of $\calJ_{\phi}$
are of Type~I, and therefore $\mu_{\phi}(\HBer)=0$.
By the final statement of \cite[Th\'{e}or\`{e}me~D]{FRL}, then, we
have
$$h_{\mu}(\phi) = h_{\textup{top}}(\phi) = \log \deg\phi =\log 6.$$
Thus, $\mu=\mu_{\phi}$ is a measure of maximal entropy for $\phi$
if $p\neq 3$.
\end{remark}

{\bf Acknowledgements.}
Authors R.B., E.K., and O.M.\
gratefully acknowledge the support of NSF grant DMS-1201341.
Authors D.B.\ and R.C.\ gratefully acknowledge the support of
Amherst College's Dean of Faculty student support funds.
Author D.O.\
gratefully acknowledges the support of Amherst College's
Schupf Scholar program.
The authors thank Mike Boyle, Charles Favre, and Cesar Silva for their helpful
discussions, and the referee for a careful reading
and helpful suggestions and corrections.


\begin{thebibliography}{99}

\bibitem{BRold}
Matthew H.\ Baker and Robert S.\ Rumely,
Equidistribution of small points, rational dynamics, and potential theory,
\emph{Ann.\ Inst.\ Fourier (Grenoble)} \textbf{56} (2006),
625-688.

\bibitem{BR}
Matthew H.\ Baker and Robert S.\ Rumely,
{\em Potential Theory and Dynamics on the Berkovich Projective Line,}
Amer.\ Math.\ Soc., Providence, 2010.


%

\bibitem{BenAZ} Robert L.~Benedetto, 
Non-archimedean dynamics in dimension one: Lecture notes,
Arizona Winter School 2010.
Available at
\texttt{http://swc.math.arizona.edu/aws/2010/2010BenedettoNotes-09Mar.pdf}.

\bibitem{BenQ}
Robert L.\ Benedetto,
A non-archimedean quartic map with connected Julia set
and non-maximal entropy,
in preparation.

\bibitem{Ber}
Vladimir G.\ Berkovich,
\newblock{\emph{Spectral Theory and Analytic Geometry over
    Non-archimedean Fields},}
\newblock{Amer.\ Math.\ Soc., Providence, 1990.}




	
\bibitem{BGR}
Siegfried Bosch, Ulrich G\"{u}ntzer, and Reinhold Remmert,
\emph{Non-Archimedean Analysis: A systematic approach to
        rigid analytic geometry},
Springer-Verlag Berlin, Heidelberg 1984.

\bibitem{BBG}
Mike Boyle, J\'{e}r\^{o}me Buzzi, and Ricardo G\'{o}mez,
Almost isomorphism for countable state Markov shifts,
\emph{J.\ Reine Angew.\ Math.}\  \textbf{592} (2006), 23--47.




\bibitem{DF1}
Laura DeMarco and Xander Faber,
Degenerations of Complex Dynamical Systems
\texttt{arxiv:1302.4679}.

\bibitem{DF2}
Laura DeMarco and Xander Faber,
Degenerations of Complex Dynamical Systems II:
Analytic and Algebraic Stability,
\texttt{arxiv:1309.7103}.


\bibitem{FJ}
Charles Favre and Mattias Jonsson,
\emph{The Valuative Tree},
Lecture Notes in Mathematics 1853,
Springer-Verlag, Berlin, 2004.

\bibitem{FRLold}
Charles Favre and Juan Rivera-Letelier,
Th\'{e}or\`{e}me d'\'{e}quidistribution de Brolin
en dynamique $p$-adique,
\emph{C.\ R.\ Math.\ Acad.\ Sci.\ Paris} \textbf{339} (2004),
271--276.

\bibitem{FRL}
Charles Favre and Juan Rivera-Letelier,
\newblock{Th\'{e}orie ergodique des fractions rationnelles
  sur un corps ultram\'{e}trique.}
\newblock{\emph{Proc.\ Lond.\ Math.\ Soc.\ (3)} \textbf{100}
  (2010), 116--154.}

\bibitem{FRL3}
Charles Favre and Juan Rivera-Letelier,
Expansion et entropie en dynamique non-archim\'{e}dienne,
in preparation, 2015.







\bibitem{Gur69}
B.\ M.\ Gurevich,
Topological entropy of enumerable Markov chains,
\emph{Dokl.\ Akad.\ Nauk SSSR} \textbf{187} (1969), 715--718;
English translation in
\emph{Soviet Math.\ Dokl.} \textbf{10} (1969), 911--915.


\bibitem{Gur70}
B.\ M.\ Gurevich,
Shift entropy and Markov measures in the math space
of a denumerable graph,
\emph{Dokl.\ Akad.\ Nauk SSSR} \textbf{192} (1970), 963--965;
English translation in
\emph{Soviet Math.\ Dokl.} \textbf{11} (1970), 744--747.

\bibitem{GurSav}
B.\ M.\ Gurevich and S.\ V.\ Savchenko,
Thermodynamic formalism for countable symbolic Markov shifts,
\emph{Uspekhi Mat.\ Nauk.} \textbf{53}(2) (1998), 3--106;
English translation in
\emph{Russian Math.\ Surv.} \textbf{53}(2) (1998), 245-344.




\bibitem{Kit}
Bruce P.\ Kitchens,
\emph{Symbolic Dynamics: One-sided, Two-sided, and Countable State
Markov Shifts},
Springer-Verlag, Berlin, 1998.

\bibitem{Kiwi}
Jan Kiwi,
Puiseux series polynomial dynamics and iteration
of complex cubic polynomials,
\emph{Ann.\ Inst.\ Fourier (Grenoble)} \textbf{56} (2006), 1337–-1404.
















\bibitem{PU}
Feliks Przytycki and Mariusz Urba\'{n}ski,
\emph{Conformal Fractals: Ergodic Theory Methods}
Cambridge University Press, Cambridge, 2010.

\bibitem{Riv1}
Juan Rivera-Letelier,
\newblock{Dynamique des fonctions rationnelles sur des corps locaux,}
\newblock{{\em Ast\'erisque} {\bf 287} (2003), 147--230.}

\bibitem{Riv2}
J.~Rivera-Letelier,
\newblock{Espace hyperbolique $p$-adique et dynamique des
        fonctions rationnelles,}
\newblock{{\em Compositio Math.} {\bf 138} (2003), 199--231.}

\bibitem{Riv3}
Juan Rivera-Letelier,
\newblock{Points p\'{e}riodiques des fonctions rationnelles dans
  l'espace hyperbolique $p$-adique,}
\newblock{\emph{Comment.\ Math.\ Helv.}\ \textbf{80} (2005), 593-–629.}

\bibitem{Rob}
Alain M.\ Robert,
{\em A Course in $p$-adic Analysis,}
Springer-Verlag, New York, 2000.

\bibitem{Roy}
Halsey L.\ Royden,
\emph{Real Analysis}, 3rd ed.,
Macmillan, New York, 1988.

\bibitem{Rue}
Silvie Ruette,
On the Vere-Jones classification and existence of maximal
measures for countable topological Markov chaains,
\emph{Pacific J.\ Math.}\ textbf{209} (2003), 365--380.

\bibitem{Sar}
Omri M.\ Sarig,
Thermodynamic formalism for countable Markov shifts,
\emph{Ergod.\ Theor.\ Dynam.\ Syst.}\ \textbf{19} (1999),
1565--1593.




\bibitem{VJ}
David Vere-Jones,
Geometric ergodicity in denumerable Markov chains,
\emph{Quart.\ J.\ Math.\ Oxford Ser.~2},
\textbf{13} (1962), 7--28.

\bibitem{Wal}
Peter Walters,
\emph{An Introduction to Ergodic Theory},
Springer-Verlag, New York, 1982.



\end{thebibliography}
\end{document}